\documentclass{amsart}
\usepackage{amsmath,amsfonts,amsthm,enumerate,color,enumerate,amscd,graphicx}
\usepackage{pdfsync}
\usepackage{mathrsfs}
\usepackage{color}
 \usepackage[colorlinks=true, pdfstartview=FitV, linkcolor=blue, citecolor=blue, urlcolor=blue]{hyperref}

\newcommand{\End}{\text{\rm End}}
\newcommand{\thmref}[1]{Theorem~\ref{#1}}

\newcommand{\secref}[1]{\S\ref{#1}}
\newcommand{\lemref}[1]{Lemma~\ref{#1}}
\newcommand{\eqnref}[1]{~(\ref{#1})}

\newcommand{\germ}{\mathfrak}

\newtheorem{thm}{Theorem}[section]
\newtheorem{lem}[thm]{Lemma}

\theoremstyle{definition}

\newtheorem{prop}[thm]{Proposition}

\theoremstyle{rem}
\numberwithin{equation}{section}

\address{Department of Mathematics,
The College of Charleston,
Charleston SC 29424}
\email{coxbl@cofc.edu, jurisiche@cofc.edu}

\address{Institute of Science and Tecnology,
Federal University of Sao Paulo,
Sao Jose dos Campos SP 12247014, Brazil}
\email{martins.renato@unifesp.br}

\usepackage{amssymb}
\title{The 3-point Virasoro algebra  and its action on a Fock space}
\author{Ben Cox}\author{Elizabeth Jurisich}\author{Renato A. Martins}

\begin{document}

\begin{abstract}
We define a 3-point Virasoro algebra, and construct a representation of it on a previously defined Fock space for the 3-point affine algebra $\mathfrak{sl}(2, \mathcal R)  \oplus\left( \Omega_{\mathcal R}/d{\mathcal R}\right)$. \end{abstract}

\maketitle
 
\section{Introduction}  

Previous work of Kassel and Loday (see \cite{MR694130}, and \cite{MR772062}) show that if $R$ is a commutative algebra and $\mathfrak g $ is a simple Lie algebra, both defined over the complex numbers, then the universal central extension $\hat{{\mathfrak g}}:=\widehat{{\mathfrak g}(R)}$ of ${\mathfrak g}\otimes R$ is the vector space $\widehat{\mathfrak{g}(R)}:=\left({\mathfrak g}\otimes R\right)\oplus \Omega_R^1/dR$ where $\Omega_R^1/dR$ is the space of K\"ahler differentials modulo exact forms (see \cite{MR772062}).  More precisely the vector space $\hat{{\mathfrak g}}$ is made into a Lie algebra by defining
$$
[x\otimes f,y\otimes g]:=[xy]\otimes fg+(x,y)\overline{fdg},\quad [x\otimes f,\omega]=0
$$
for all $x,y\in\mathfrak g$, $f,g\in R$,  $\omega\in \Omega_R^1/dR$.  In the above $(-,-)$ denotes the Killing form on $\mathfrak g$ and $\overline{a}$ denotes the congruence class of $a\in\Omega^1_R$ modulo $dR$.    A natural question
is whether there exists free field or Wakimoto type realizations of these algebras.  From the work of Wakimoto, and Feigin and Frenkel  the answer is known when $R=\mathbb C[t^{\pm 1}]$ is the ring of Laurent polynomials in one variable (see \cite{W} and \cite{MR92f:17026}). Before describing the current work, we review a few other rings $R$ for which there is a known free field type realization.
  
  The initial motivation for the use of Wakimoto's realization was to prove a conjecture of Kac and Kazhdan involving the character of certain irreducible representations of affine Kac-Moody algebras at the critical level (see \cite{W} and \cite{MR2146349}). Another motivation for constructing such free field realizations is that they have been used to describe integral solutions to the Knizhnik-Zamolodchikov equations (see for example \cite{MR1077959} and \cite{MR1629472} and their references).  A third is that they are used in determining the center of a certain completion of the enveloping algebra of an affine Lie algebra at the critical level, which is an important ingredient in the geometric Langland's correspondence \cite{MR2332156}.  Yet another motivation is that Wakimoto realizations of an affine Lie algebra appear naturally in the context of the generalized AKNS hierarchies \cite{MR1729358}.

   In Kazhdan and Luszig's explicit study of the tensor structure of modules for affine Lie algebras (see \cite{MR1186962} and \cite{MR1104840}) the ring of functions on the Riemann sphere regular everywhere except at a finite number of points appears naturally.   This algebra is called by Bremner the {\it $n$-point algebra}.  One can find in the book  \cite[Ch. 12]{MR1849359} algebras of the form $\oplus _{i=1}^n\mathfrak g((t-x_i))\oplus\mathbb Cc$ appearing in the description of the conformal blocks.  These contain the $n$-point algebras $ \mathfrak g\otimes \mathbb C[t, (t-x_1)^{-1},\dots, (t-x_N)^{-1}]\oplus\mathbb Cc$ modulo part of the center $\Omega_R/dR$. In \cite{MR1261553} Bremner explicitly described the universal central extension of such an algebra in terms of a basis.   In \cite{MR3211093} the authors give an explicit description of the two cocyles and hence the universal central extension of what is called the $n$-point Virasoro algebra and its action on modules of densities.    The {\it $4$-point ring} is $R=R_a=\mathbb C[s,s^{-1},(s-1)^{-1},(s-a)^{-1}]$ where $a\in\mathbb C\backslash\{0,1\}$.    Set $S:=S_b=\mathbb C[t,t^{-1},u]$ where $u^2=t^2-2bt+1$ with $b$ a complex number not equal to $\pm 1$.  After observing $R_a\cong S_b$;
  Bremner gave an explicit  description of the universal central extension of $\mathfrak g\otimes S_b$, in terms of ultraspherical (Gegenbauer) polynomials (see \cite{MR1249871}).     The first author of this present article gave in  \cite{MR2373448} a realization of the four point algebra in terms of infinite sums of partial differential operators  acting on a polynomial ring in infinitely many variables and where the center acts nontrivially.  See also \cite{MR1303073}), \cite{FailS}, \cite{MR2183958} and \cite{MR2541818}) for work on other rings besides the $n$-point algebras.

 Below we study the three point algebra case where $R$ denotes the ring of rational functions with poles only in the set $\{a_1,a_2,a_3\}$. This algebra is isomorphic to $\mathbb C[s,s^{-1},(s-1)^{-1}]$. Schlichenmaier has a somewhat different description of the three point algebra as having coordinate ring $\mathbb C[(z^2-a^2)^k,z(z^2-a^2)^k\,|\, k\in\mathbb Z]$ where $a\neq 0$ (see \cite{MR2058804}). 
In \cite{MR3245847} it was noted that $R\cong \mathbb C[t,t^{-1},u\,|\,u^2=t^2+4t]$, and thus the three point algebra resembles $S_b$ above.
 Besides Bremner's article mentioned above, other work on the universal central extension of $3$-point algebras can be found in \cite{MR2286073}.   This article is restricted to the representation theory of the affine $3$-point algebra and its algebra of derivations mainly to simplify calculations.

The main result of \cite{MR3245847}, reviewed below in \thmref{mainresult0} provides a natural free field realization in terms of a $\beta$-$\gamma$-system and the three point Heisenberg algebra, of the three point affine Lie algebra when $\mathfrak g=\mathfrak{sl}(2,\mathbb C)$.   Just as in the case of intermediate Wakimoto modules defined in
\cite{ MR2271362}, there are two different realizations given by a parameter $r=0,1$ of this action on a Fock space $\mathcal F$ depending on two different normal orderings.   When $r=1$ we get a free field realization and when $r=0$ we obtain a realization in terms of infinite sums of partial differential operators on polynomial rings in infinitely many variables.

In \secref{3pt2cocycle} we rewrite the two cocycles given in \cite{MR3211093} used to define the three point Virasoro algebra $\mathfrak V$,  using a basis of $R=\mathbb C[t^{\pm 1},u\,|\,u^2=t^2+4t]$ rather than a basis of $S=\mathbb C[s^{\pm 1},(s-1)^{-1}]$.  
The advantage of using the ring $R$ is that the generating fields and their relations for $\text{Der}\,(R)$ can be written in a fairly simple and compact fashion, see \eqnref{eezw}-\eqnref{dezw}.  One of the problems listed in \cite{MR1109216} is to describe the universal central extension of the three point Witt algebra which we give 
in \eqnref{phiee}-\eqnref{phidd}.     A variation of this is also given in \cite{MR3211093}.

The central result  of this current article, \thmref{mainresult}, provides a natural action of this three point Virasoro algebra $\mathfrak V$, on the realization for the three point current algebra $\widehat{\mathfrak{sl}_2(R)}$, given in \thmref{mainresult0}.  As for the current type algebra $\mathfrak{sl}_2(\mathbb C)\otimes R$ these realizations of  $\mathfrak V$ depend on a normal ordering parametrized by $r=0,1$.   The proof is based on Wick's Theorem and Taylor's Theorem given in the context of vertex operator algebras.  For the reader's convenience these theorems are stated in the appendix.   We conjecture that the semi-direct product of the three point Virasoro algebra with the three point current algebra acts on the free field realization provided $r=1$.    This semi-direct product can be thought of as a kind of gauge algebra and will be studied in a future paper.

 The simplest non-trivial example of a Krichever-Novikov algebra beyond an affine Kac-Moody algebra 
 (see \cite{MR902293}, \cite{MR925072}, \cite{MR998426}) is perhaps the  three point algebra.  On the other hand interesting and 
foundational work has be done by Krichever, Novikov, Schlichenmaier, and Sheinman on the representation theory of the Krichever-Novikov algebras. 
 In particular Wess-Zumino-Witten-Novikov theory and analogues of the Knizhnik-Zamolodchikov equations are developed for  these algebras 
(see the survey article \cite{MR2152962}, and for example \cite{MR2058804},  \cite{MR1989644}, \cite{MR1666274}, \cite{MR1706819}, and \cite{MR2072650}).  In the affine Lie algebra setting such as $\widehat{\mathfrak{sl}_2(\mathbb C)}$, the integral solutions are described in terms of hypergeometric functions.
 We plan to use the realization given by \thmref{mainresult0} and \thmref{mainresult} in future work to arrive at an explicit description of the corresponding Knizhnik-Zamolodchikov equations with the goal of providing integral solutions of these equations for the three point algebra.  

\section{Preliminary material  and Notation}

All vector spaces and algebras are over $\mathbb C$. All power series are formal series. 
\subsection{Formal Distributions}
We need recall notation that will simplify many of the arguments
made later.
This notation follows roughly \cite{MR99f:17033} and \cite
{MR2000k:17036}:  The {\it formal delta function}
$\delta(z/w)$ is the formal distribution
$$
\delta(z/w)=z^{-1}\sum_{n\in\mathbb Z}z^{-n}w^{n}=w^{-1}\sum_{n\in\mathbb Z}z^{n}w^{-n}.
$$
For any sequence of elements $\{a_{m}\}_{m\in
\mathbb Z}$ in the ring $\End (V)$, $V$ a vector space,  the
formal distribution
\begin{align*}
a(z):&=\sum_{m\in\mathbb Z}a_{(m)}z^{-m-1}
\end{align*}
is called a {\it field}, if for any $v\in V$, $a_{m}v=0$ for $m\gg0$.
If $a(z)$ is a field, then we set
\begin{align}\label{usualnormalordering}
    a(z)_-:&=\sum_{m\geq 0}a_{(m)}z^{-m-1},\quad\text{and}\quad
   a(z)_+:=\sum_{m<0}a_{(m)}z^{-m-1}.
\end{align}
 The {\it normal ordered product} of two distributions
$a(z)$ and
$b(w)$ (and their coefficients) is
defined by
\begin{equation}\label{normalorder}
\sum_{m\in\mathbb Z}\sum_{n\in\mathbb
Z}:a_{(m)}b_{(n)}:z^{-m-1}w^{-n-1}=:a(z)b(w):=a(z)_+b(w)+b(w)a(z)_-.
\end{equation}

Now we should point out that while $:a^1(z_1)\cdots a^m(z_m):$
is always defined as a formal series, we will only define $:a(z)
b(z)::=\lim_{w\to z}:a(z)b(w):$ 
for certain pairs
$(a(z),b(w))$.  

Then one defines recursively
\[
:a^1(z_1)\cdots a^k(z_k):=:a^1(z_1)\left(:a^2(z_2)\left(:\cdots
:a^{k-1}(z_{k-1}) a^k(z_k):\right)\cdots
:\right):,
\]
while normal ordered product
\[
:a^1(z)\cdots
a^k(z):=\lim_{z_1,z_2,\cdots, z_k\to
z} :a^1(z_1)\left(:a^2(z_2)\left(:\cdots :a^{k-1}(z_{k-1})
a^k(z_k):\right)\cdots
\right):
\]
will only be defined for certain $k$-tuples $(a^1,\dots,a^k)$.

Let 
\begin{equation}\label{contraction}
\lfloor
ab\rfloor=a(z)b(w)-:a(z)b(w):= [a(z)_-,b(w)],
\end{equation}
(half of
$[a(z),b(w)]$) denote the {\it contraction} of any two formal distributions 
$a(z)$ and $b(w)$. Note that the the variables $z$,$w$ are usually suppressed in this notation, when no confusion will arise.  

 \section{Oscillator algebras}\label{oscilllatoralg}
 \subsection{The $\beta-\gamma$ system}   The following construction in the physics literature is often called the $\beta-\gamma$ system which corresponds to our $a$ and $a^*$ below.
 Let $\hat{\mathfrak a}$ be the infinite dimensional oscillator algebra with generators $a_n,a_n^*,a^1_n,a^{1*}_n,\,n\in\mathbb Z$ together with $\mathbf 1$ satisfying the relations 
\begin{gather*}
[a_n,a_m]=[a_m,a_n^1]=[a_m,a_n^{1*}]=[a^*_n,a^*_m]=[a^*_n,a^1_m]=[a^*_n,a^{1*}_{m}]=0,\\
[a_n^{1},a_m^{1}]=[a_n^{1*},a_m^{1*}]=0=[\mathfrak a,\mathbf 1], \\
[a_n,a_m^*]=\delta_{m+n,0}\mathbf 1=[a^1_n,a_m^{1*}].
\end{gather*}
For  $c=a,a^1$ and respectively $X=x,x^1$ with $r=0$ or $r=1$, we define $\mathbb C[\mathbf x]:= \mathbb C[x_n,x_n^1\,|\,n\in\mathbb Z]$ and $\rho:\hat{\mathfrak a}\to \mathfrak{gl}(\mathbb C[\mathbf x])$ by
\begin{align}
\rho_r( c_{m}):&=\begin{cases}
  \partial/\partial
X_{m}&\quad \text{if}\quad m\geq 0,\enspace\text{and}\enspace  r=0
\\ X_{m} &\quad \text{otherwise},
\end{cases}\label{c}
 \\
\rho_r(c_{m}^*):&=
\begin{cases}X_{-m} &\enspace \text{if}\quad m\leq
0,\enspace\text{and}\enspace r=0\\ -\partial/\partial
X_{-m}&\enspace \text{otherwise}. \end{cases}\label{c*}
\end{align}
and $\rho_r(\mathbf 1)=1$.
These two representations can be constructed using induction:
For $r=0$ the representation 
$\rho_0$ is the
$\hat{\mathfrak a}$-module generated by $1=:|0\rangle$, where
$$
a_{m}|0\rangle=a^1_{m}|0\rangle=0,\quad m\geq  0,
\quad a_{m}^*|0\rangle= a_{m}^{1*}|0\rangle=0,\quad m>0.
$$
For $r=1$ the representation 
$\rho_1$ is the
$\hat{\mathfrak a}$-module generated by $1=:|0\rangle$, where
$$
\quad a_{m}^*|0\rangle= a_{m}^{1*}|0\rangle=0,\quad m\in\mathbb Z.
$$
If we define
\begin{equation}
 \alpha(z):=\sum_{n\in\mathbb Z}a_nz^{-n-1},\quad  \alpha^*(z):=\sum_{n\in\mathbb Z}a_n^*z^{-n}, \label{alpha}
\end{equation}
and
\begin{equation}
 \alpha^1(z):=\sum_{n\in\mathbb Z}a^1_nz^{-n-1},\quad  \alpha^{1*}(z):=\sum_{n\in\mathbb Z}a^{1*}_nz^{-n}, \label{alpha1}
\end{equation}
then 
\begin{align*}
[\alpha(z),\alpha(w)]&=[\alpha^*(z),\alpha^*(w)]=[\alpha^{1}(z),\alpha^{1}(w)]=[\alpha^{1*}(z),\alpha^{1*}(w)]=0  \\
[\alpha(z),\alpha^*(w)]&=[\alpha^1(z),\alpha^{1*}(w)]
    =\mathbf 1\delta(z/w).
\end{align*}
Observe that $\rho_1(\alpha(z))$ and 
$\rho_1(\alpha^1(z))$ are not fields whereas $\rho_r(\alpha^*(z))$ and $\rho_r(\alpha^{1*}(z))$
are always fields.   
Corresponding to these two representations there are two possible normal orderings:  For $r=0$ we use the usual normal ordering given by \eqnref{usualnormalordering} and for $r=1$ we define the {\it natural normal ordering} to be 
\begin{alignat*}{2}
\alpha(z)_+&=\alpha(z),\quad &\alpha(z)_-&=0 \\
\alpha^1(z)_+&=\alpha^1(z),\quad &\alpha^1(z)_-&=0 \\
\alpha^*(z)_+&=0,\quad &\alpha^*(z)_-&=\alpha^*(z), \\
\alpha^{1*}(z)_+&=0,\quad &\alpha^{1*}(z)_-&=\alpha^{1*}(z) ,
\end{alignat*}

This means in particular that for $r=0$ we get 
\begin{align}
\lfloor \alpha \alpha^*  \rfloor =\lfloor \alpha(z) ,\alpha^*(w) \rfloor
&=\sum_{m\geq 0} \delta_{m+n,0}z^{-m-1}w^{-n}
=\delta_-(z/w)
=\ 
\,\iota_{z,w}\left(\frac{1}{z-w}\right)\\
\lfloor \alpha^* \alpha \rfloor
&
=-\sum_{m\geq 1} \delta_{m+n,0}z^{-m}
w^{-n-1}
=-\delta_+(w/z)=\,\iota_{z,w}\left(\frac{1}{w-z}
\right)
\end{align}
(where $\iota_{z,w}$ denotes Taylor series expansion in the ``region'' $|z|>|w|$), 
and for $r=1$ 
\begin{align}
\lfloor \alpha ,\alpha^* \rfloor
&=[\alpha(z)_-,\alpha^*(w)]=0 \\
\lfloor \alpha^*  \alpha \rfloor
&=[\alpha^*(z)_-,\alpha(w)]=
-\sum_{\in\mathbb Z} \delta_{m+n,0}z^{-m}
w^{-n-1}
=- \delta(w/z),
\end{align}
where similar results hold for $\alpha^1(z)$.
Notice that in both cases we have
$$
[\alpha(z),\alpha^*(w)]=
\lfloor \alpha(z)\alpha^*(w)\rfloor-\lfloor\alpha^*(w) \alpha(z)\rfloor=\delta(z/w).
$$

Recall that the singular part of the {\it operator product
expansion}
$$
\lfloor
ab\rfloor=\sum_{j=0}^{N-1}\iota_{z,w}\left(\frac{1}{(z-w)^{j+1}}
\right)c^j(w)
$$
completely determines the bracket of mutually local formal
distributions $a(z)$ and $b(w)$.   (See Theorem \ref{Kacsthm} of the Appendix). One writes
$$
a(z)b(w)\sim \sum_{j=0}^{N-1}\frac{c^j(w)}{(z-w)^{j+1}}.
$$

\section{The $3$-point algebras.}\label{threepoint}

\subsection{Three point rings} There are at least four incarnations of the three point algebra each depending on the coordinate ring:
  Fix $0\neq a\in\mathbb C$.
Let \begin{align*}
\mathcal S&:=\mathbb C[s,s^{-1},(s-1)^{-1} ] ,\\
 \mathcal R&:=\mathbb C[t,t^{-1},u\,|\, u^2=t^2+4t], \\
 \mathcal A&:=\mathcal A_a=\mathbb C[(z^2-a^2)^k,z(z^2-a^2)^j\,|\,k,j\in\mathbb Z].
 \end{align*}
M. Bremner introduced the ring $\mathcal S$ and M. Schlichenmaier introduced $\mathcal A$ (see \cite{MR2058804}).  Variants of $\mathcal R$ were introduced by Bremner for elliptic and $4$-point algebras.
\begin{prop}[\cite{MR3245847}]
\begin{enumerate}
\item  The map $t\mapsto s^{-1}(s-1)^2$, and $u\mapsto s-s^{-1}$, defines an isomorphism of $\mathcal R$ and $\mathcal S$ .
\item The rings $\mathcal A$ and $\mathcal R$ are isomorphic.
\end{enumerate}
\end{prop}
The fourth incarnation appears in the work of Benkart and Terwilliger given in terms of the tetrahedron algebra (see \cite{MR2286073}) . 
We will only work with $R=\mathcal R$.  
\subsection{The Universal Central Extension of the Current Algebra $\mathfrak g\otimes R$.}
 Suppose $R$ is a commutative algebra defined over $\mathbb C$.
Consider the left $R$-module  $F=R\otimes R$ with left action given by $f( g\otimes h ) = f g\otimes h$ for $f,g,h\in R$ and let $K$  be the submodule generated by the elements $1\otimes fg  -f \otimes g -g\otimes f$.   Then $\Omega_R^1=F/K$ is the module of {\it K\"ahler differentials}.  The element $f\otimes g+K$ is traditionally denoted by $fdg$.  The canonical map $d:R\to \Omega_R^1$ is given by $df  = 1\otimes f  + K$.  The {\it exact differentials} are the elements of the subspace $dR$.  The coset  of $fdg$  modulo $dR$ is denoted by $\overline{fdg}$.  C. Kassel proved that the universal central extension of the current algebra $\mathfrak g\otimes R$ where $\mathfrak g$ is a simple finite dimensional Lie algebra defined over $\mathbb C$, is the vector space $\hat{\mathfrak g}=(\mathfrak g\otimes R)\oplus \Omega_R^1/dR$ with Lie bracket given by
$$
[x\otimes f,Y\otimes g]=[xy]\otimes fg+(x,y)\overline{fdg},  [x\otimes f,\omega]=0,  [\omega,\omega']=0,
$$
  where $x,y\in\mathfrak g$, and $\omega,\omega'\in \Omega_R^1/dR$ and $(x,y)$  denotes the Killing  form  on $\mathfrak g$.

\begin{prop}[\cite{MR3245847}, \cite{MR1261553}, see also \cite{MR1249871}]\label{uce}  Let $\mathcal R$ be as above.  The set 
$$
\{\omega_0:=\overline{t^{-1} dt},\enspace \omega_1:=\overline{t^{-1}u\,dt}\}
$$
 is a basis of $\Omega_\mathcal R^1/d\mathcal R$.
\end{prop}

\begin{thm}[\cite{MR3245847}]\label{3ptthm}

The universal central extension $\hat{ \mathfrak g}$ of the algebra $\mathfrak{sl}(2,\mathbb C)\otimes \mathcal R$ is isomorphic to the Lie algebra with generators $e_n$, $e_n^1$, $f_n$, $f_n^1$, $h_n$, $h_n^1$, $n\in\mathbb Z$, $\omega_0$, $\omega_1$ and relations given by

\begin{align*}
[x_m,x_n]&:=[x_m,x_n^1]=[x_m^1,x_n^1]=0,\quad \text{ for }x=e,f \\ 
[h_m,h_n]&: =(n-m)\delta_{m,-n}\omega_0 , \quad
[h^1_m,h^1_n] := (n-m)\left(\delta_{m+n,-2}+4\delta_{m+n,-1}\right)\omega_0, \\
[h_m,h_n^1]&:=-2\mu_{m,n} \omega_1 ,\\
 [\omega_i,x_m]&=[\omega_i,\omega_j]=0,\quad \text{ for }x=e,f,h,\quad i,j\in\{0,1\} \\
[e_m,f_n]&:=h_{m+n}-m\delta_{m,-n}\omega_0, \quad
[e_m,f_n^1] :=h^1_{m+n}-m\mu_{m,n} \omega_1=:[e_m^1,f_n], \\
[e_m^1,f_n^1]&:=h_{m+n+2}+4h_{m+n+1}+\frac{1}{2}(n-m)\left(\delta_{m+n,-2}+4\delta_{m+n,-1}\right)\omega_0,  \\
[h_m,e_n]&:=2e_{m+n},\quad 
[h_m,e^1_n] :=2e^1_{m+n} = :[h_m^1,e_m],  \\
[h_m^1,e_n^1]&:=2e_{m+n+2} +8e_{m+n+1}, \quad 
[h_m,f_n] :=-2f_{m+n}, \\
[h_m,f^1_n]&:=-2f^1_{m+n} =:[h_m^1,f_m], \quad 
[h_m^1,f_n^1] :=-2f_{m+n+2} -8f_{m+n+1} , 
\end{align*}
for all $m,n\in\mathbb Z$, where $\mu_{m,n}:= m\frac{(-1)^{m+n + 1}2^{m+n}(2(m+n)-1)!!}{(m+n+1)!}$.
\end{thm}

For $m=i-\frac{1}{2}$, $i\in\mathbb Z+\frac{1}{2}$ and $x\in\mathfrak {sl}(2,\mathbb C)$., define $x_{m+\frac{1}{2}}=x\otimes t^{i-\frac{1}{2}}u=x^1_m$ and $x_m:=x\otimes t^m$.  Motivated by conformal field theory we set
\begin{align*}
x^1(z)
&:=\sum_{m\in\mathbb Z}x_{m+\frac{1}{2}}z^{-m-1},\quad x(z):=\sum_{m\in\mathbb Z}x_{m}z^{-m-1}.
\end{align*}
Then the relations in \thmref{3ptthm} correspond to 
\begin{align}
[x(z),y(w)]
&= [xy](w)\delta(z/w)-(x,y)\omega_0\partial_w\delta(z/w), \label{r1} \\ 
[x^1(z),y^1(w)]
&= P(w)\left([x,y](w)\delta(z/w) -(x,y)\omega_0\partial_w\delta(z/w)\right)\\
    &  -\frac{1}{2}(x,y)(\partial_w P)(w)\omega_0  \delta(z/w),  \label{r2} \\ 
[x(z),y^1(w)]
&=[x,y]^1(w)\delta(z/w) -\frac{1}{2}(x,y)\omega_1\sqrt{1+(4/w)}w\partial_w\delta(z/w)\\ 
& = [x^1(z),y(w)], \label{r3}
 \end{align}
where $x,y\in\{e,f,h\}$

\subsection{The $3$-point Heisenberg algebra} 
The Cartan subalgebra $\mathfrak h$ tensored with $\mathcal R$ generates a subalgebra of $\hat{{\mathfrak g}}$ which is an extension of an oscillator algebra.    This extension motivates the following definition:  The Lie algebra with generators $b_{m},b_m^1$, $m\in\mathbb Z$, $\mathbf 1_0,\mathbf 1_1 $, and relations
\begin{align*}
[b_{m},b_{n}]&=(n-m)\,\delta_{m+n,0}\mathbf 1_0=-2m\,\delta_{m+n,0}\mathbf 1_0 \\
[b^1_m,b_n^1] &=(n-m)\left(\delta_{m+n,-2}+4\delta_{m+n,-1}\right)\mathbf 1_0   \\ &= 2 \left((n+1)\delta_{m+n,-2}+(4n+2)\delta_{m+n,-1}\right)\mathbf 1_0 \notag \\
[b^1_m,b_n] &=2\mu_{m,n}\mathbf 1_1 = -[b_n, b^1_m]
  \\
[b_{m},\mathbf 1_0]&=[b_{m}^1,\mathbf 1_0]=[b_{m},\mathbf 1_1]=[b_{m}^1,\mathbf 1_1]= 0. 
\end{align*}
is the {\it $3$-point (affine) Heisenberg algebra} which we denote by $\hat{\mathfrak h}_3$.

If we introduce the formal distributions
\begin{equation} 
\beta(z):=\sum_{n\in\mathbb Z} b_nz^{-n-1},\quad \beta^1(z):=\sum_{n\in\mathbb Z}b_n^1z^{-n-1}=\sum_{n \in\mathbb Z}b_{n+\frac{1}{2}}z^{-n-1}.
\end{equation}
(where $b_{n+\frac{1}{2}}:=b^1_n$)
then the relations above can be rewritten in the form
\begin{align*}\label{bosonrelations}
[\beta(z),\beta(w)]&=2\mathbf 1_0\partial_z\delta(z/w)=-2\partial_w\delta(z/w)\mathbf 1_0 \\
[\beta^1(z),\beta^1(w)]
&=-2\left((w^2+4w) \partial_w(\delta(z/w)+ (2+w) \delta(z/w)\right)\mathbf 1_0 \\
[\beta(z),\beta^1(w)]&= -\sqrt{1+(4/w)}w \partial_w\delta(z/w)\mathbf 1_1
\end{align*}

 Set
\begin{align*}
\hat{\mathfrak h}_3^\pm:&=\sum_{n\gtrless 0}\left(\mathbb Cb_n+\mathbb Cb_n^1\right),\quad
\hat{ \mathfrak h}_3^0:=  \mathbb C\mathbf 1_0\oplus \mathbb C\mathbf 1_1\oplus \mathbb Cb_0\oplus \mathbb Cb^1_0.
\end{align*}
We introduce a Borel type subalgebra
\begin{align*}
\hat{\mathfrak b}_3&= \hat{\mathfrak h}_3^+\oplus \hat{\mathfrak h}_3^0.
\end{align*}
That $\hat{\mathfrak b}_3$ is a subalgebra follows from the above defining relations.

\begin{lem}\label{heisenbergprop}
Let $\mathcal V=\mathbb C\mathbf v_0\oplus \mathbb C\mathbf v_1$ be a two dimensional representation of 
$\hat{\mathfrak h}_3^+ $ with $\hat{\mathfrak h}_3^+\mathbf v_i=0$ for $i=0,1$.   Fix  $ B_0, B^1_{i,j}$ for $i,j = 0,1$ with $B^1_{00} = B^1_{11}$ and $ \chi_1,\kappa_0 \in \mathbb C$  and let 
\begin{align*}
b_0\mathbf v_0&=B_0 \mathbf v_0,  &b_0\mathbf v_1&=B_0 \mathbf v_1 \\
b_0^1\mathbf v_0&=B^1_{00} \mathbf v_0+B^1_{01}\mathbf v_1,  &b_0^1\mathbf v_1&=B^1_{10} \mathbf v_0+B^1_{11}\mathbf v_1\\
\mathbf 1_1\mathbf v_i&=\chi _1  \mathbf v_i,\quad   &\mathbf 1_0\mathbf v_i&=\kappa _0\mathbf v_i,\quad i=0,1.
\end{align*}
When $\chi_1$ acts as zero, the above defines a representation of  $\hat{\mathfrak b}_3$ on $\mathcal V$. 
\end{lem}

Let $ \mathbb C[\mathbf y]:= \mathbb C[y_{-n}, y_{-m}^1 | m,n \in \mathbb N^*] $. The following is a straightforward computation, with corrections to the version in \cite{MR3245847} (where some formulas for the 4-point algebra were inadvertently included).

\begin{lem} [\cite{MR3245847}] \label{rhorep}The linear map $\rho:\hat{\mathfrak b}_3\to \text{End}(\mathbb C[\mathbf y]\otimes \mathcal V)$ defined  by 
\begin{align*}
\rho(b_{n})&=y_{n} \quad \text{ for }n<0 \\
\rho(b_{n}^1)&=y_{n}^1\quad \text{ for }n<0 \\
\rho (b_n) &= -n 2 \partial_{ y_{-n} }\kappa_0   \quad \text{ for }n>0 \\
\rho(b^1_n)&= -(2+ 2n) \partial_{y^1_{-2-n}} \kappa_0 -4(1+2n) \partial_{y^1_{-1-n}} \kappa_0 \quad \text{ for }n>0\\
\rho(b^1_0)&= -2 \partial_{y^1_{-2}} \kappa_0 -4 \partial_{y^1_{-1}} \kappa_0 +B_0^1 \\
\rho(b_{0})&=B_0
\end{align*}
is a representation of $\hat{\mathfrak b}_3$.
\end{lem}

\section{The Fock space representation of the  $3$-point algebra $\hat{{\mathfrak g}}$}
 
 We recall the definition of the three point algebra and two representations constructed in \cite{MR3245847}.
 Assume that $\chi_0\in\mathbb C$ and define $\mathcal V$ as in \lemref{heisenbergprop}.
 Set
 \begin{equation}
 P(z)=z^2+4z.
 \end{equation} 
  The $\alpha(z),\alpha^1(z),\alpha^* (z)$ and $\alpha^{1*}(z)$ are generating series of oscillator algebra elements as in \eqref{alpha} and \eqref{alpha1}. Our main result in \cite{MR3245847} is the following 

\begin{thm}[\cite{MR3245847}] \label{mainresult0}  Fix $r\in\{0,1\}$, which then fixes the corresponding normal ordering convention defined in the previous section.  Set $\hat{{\mathfrak g}} =\left(\mathfrak{sl}(2,\mathbb C)\otimes \mathcal R\right)\oplus \mathbb C\omega_0\oplus \mathbb C\omega_1$. Then using \eqnref{c}, \eqnref{c*} and \lemref{rhorep}, the following defines a representation of the three point algebra $\hat{\mathfrak g}$ on $\mathbb C[\mathbf x]\otimes \mathbb C[\mathbf y]\otimes \mathcal V$:
\begin{align*}
\tau(\omega_1)&=0, \qquad
\tau(\omega_0)=\chi_0=\kappa_0+4\delta_{r,0} ,  \\ 
\tau(f(z))&=-\alpha(z), \qquad
\tau(f^1(z))=- \alpha^1(z),   \\ \\
\tau(h(z))
&=2\left(:\alpha(z)\alpha^*(z):+:\alpha^1(z)\alpha^{1*}(z): \right)
     +\beta(z) , \\  \\
\tau(h^1(z))
&=2\left(:\alpha^1(z)\alpha^*(z): +P(z):\alpha(z)\alpha^{1*}(z): \right) +\beta^1(z),  \\  \\
\tau(e(z)) 
&=:\alpha(z)(\alpha^*(z))^2 :+P(z):\alpha(z)(\alpha^{1*}(z))^2: +2 :\alpha^1(z)\alpha^*(z)\alpha^{1*}(z): \\
&\quad +\beta(z)\alpha^*(z)+\beta^1(z)\alpha^{1*}(z) +\chi_0\partial\alpha^* (z) \\ \\
\tau(e^1(z))  
&=:\alpha^1(z)(\alpha^*(z))^2:
 	+P(z)\left(:\alpha^1(z) (\alpha^{1*} (z))^2: +2 : \alpha (z)\alpha^{*} (z)\alpha^{1*}(z):\right)  \\
&\quad +\beta^1(z) \alpha^* (z)+P(z)\beta(z) \alpha^{1*}(z) +\chi_0\left(P(z)  \partial_z\alpha^{1*}(z)   +(z+2)   \alpha^{1*}(z) \right) .
\end{align*}
\end{thm}

\section{The 3-point Witt algebra}

We now restrict to the three point algebra case:
Fix the following basis elements of $\text{Der}_{\mathbb C}R$:
\begin{equation}\label{basis}
 d_n:=t^nuD,\quad   d^1_n=t^nD\quad \text{for} \quad D= (t+2)\frac{\partial}{\partial u}+u\frac{\partial}{\partial t}
\end{equation}

We call $\text{Der}_{\mathbb C}R$ the 3-point Witt algebra, for the above choice of basis vectors have relations analogous to those for the Witt algebra as seen in the following lemma.

\begin{lem} \label{commutators} The basis elements listed above for $\text{Der}_{\mathbb C}R$ satisfy the relations
\begin{align*}
[  d_m,  d_n]
&=(n-m)\left(  d_{m+n+1}+4  d_{m+n}\right) \\
[  d^1_m,  d^1_n]&=(n-m)  d_{m+n-1} \\
[  d_m,  d^1_n]&=(n-m-1)  d^1_{m+n+1}+(4n-4m-2)  d^1_{m+n}.
\end{align*}

\end{lem}
\begin{proof}
Straightforward calculations prove the commutation relations, note that 
$$
D(t^n)=\left((t+2)\frac{\partial}{\partial u}+u\frac{\partial}{\partial t} \right)(t^n)=nt^{n-1}u.
$$
So for example, recalling that $u^2 = t^2+4t$:
\begin{align*}
[  d_m, d_n]&=\left[t^muD,t^nuD\right]=(t^muD(t^nu)-t^nuD(t^mu))D \\
&=(t^mu(nt^{n-1}(t^2+4t)+(t+2)t^n)-t^nu(mt^{m-1}(t^2+4t)+(t+2)t^m))D \\
&=(n-m)t^{m+n-1}(t^2+4t)uD \\
\end{align*}
and the other relations follow similarly.
\end{proof}

\section{Representations $U_\alpha$  of the $3$-point Witt algebra $\text{Der}_{\mathbb C}R$}

Fix a complex number $\alpha$ and let $U_\alpha$ be the vector spaces with basis 
\begin{equation}
\{\mathbf a_k,\bar{\mathbf a}_k\,|\, k\in\mathbb Z\}.
\end{equation}
The action given below is motivated by viewing $U_\alpha$ as the space of formal powers of the form $t^{\alpha+k}, t^{\alpha+i}u$ with $k\in\mathbb Z$.  
\begin{lem}  The vector space
 $U_\alpha$ becomes a representation of $\text{Der}_{\mathbb C}R$ if we define the action by
 \begin{align}
 d_n\mathbf a_i&= (\alpha+i)\left(\mathbf a_{i+n+1}+4\mathbf a_{i+n }\right) \label{i} \\
  d_n\bar{\mathbf a}_i&= (\alpha+i+1)\bar{\mathbf a}_{n+i+1}+(4\alpha+4i+2)\bar{\mathbf a}_{n+i}\label{ii}  \\
  d^1_n\mathbf a_i&= (\alpha+i)\bar{\mathbf a}_{n+i-1} \label{iii} \\
 d^1_n\bar{\mathbf a}_i&= (\alpha+i+1)\mathbf a_{n+i+1}+2(2\alpha+2i+1)\mathbf a_{n+i} \label{iv}
 \end{align}
 \end{lem}
  
\begin{proof}
The result follows from verifying the relations on each type of basis element. 
\end{proof}

\section{The derivation algebra of superelliptic curves.}
A curve of the form $u^m=P(t)$ where $P(t)\in\mathbb C[t]$ is a separable polynomial and $m\geq 2$, is called a superelliptic curve. 
\begin{lem}  Let 
$R'=\mathbb C[t,t^{-1},u]$ and let $\mathfrak a$ be the ideal generated by $u^m-P(t)$ where $P(t)$ is a polynomial in $t$ and $m$ is a positive integer greater than one.  Consider the two derivations of $R'$:
\begin{align*}
D_1=\frac{P'(t)}{m}\frac{\partial }{\partial u}+u^{m-1}\frac{\partial}{\partial t},\quad
 D_2=\frac{uP'(t)}{m}\frac{\partial }{\partial u}+P(t)\frac{\partial}{\partial t}.
\end{align*}
Then $D_i(\mathfrak a)\subset \mathfrak a$ for $i=1,2$. 
\end{lem}

 Hence $D_1$ and $D_2$ descend to derivations of ring $R'/\mathfrak a$ which we still denote by $D_1$ and $D_2$.  Moreover $D_2(r)=uD_1(r)\mod \mathfrak a$ for all $r\in R'$.
  
\begin{lem}  Let  $R:=R'/\mathfrak a=\mathbb C[t,t^{-1},u|u^m=P(t)]$ where $P(t)$ and $P'(t)$ are relatively prime and $m$ is a positive integer greater than one, then
$$
\text{Der}_{\mathbb C}R=R\left( P'(t) \frac{\partial }{\partial u}+mu^{m-1}\frac{\partial}{\partial t}\right).
$$  
\end{lem}

Here are four examples of such algebras and coordinate rings $R$ with $m=2$ that appear in the literature:
\begin{enumerate}
\item The three point algebras have $P(t)=t^2+4t$. 
\item Four point algebras have $P(t)=t^2-2bt+1$, $b\neq \pm 1$.
\item Elliptic affine algebras have $P(t)=t^3-2bt^2+t$, $b\neq \pm 1$.
\item The coordinate algebras with $P(t)=(t^2-b^2)(t^2-c^2)$ appear in the work of Date, Jimbo, Kashiwara et al on Landau-Lipschitz differential equations,  $b\neq \pm c$ and $bc\neq 0$. If these extra conditions hold for $b$ and $c$, then $P(t)$ and $P'(t)=2t(2t^2-b^2-c^2)$ are relatively prime.
\end{enumerate}

\section{3-point Witt algebra representation}
We now construct a representation using the oscillator algebra. 
Define $\pi:\text{Der}(R)\to \text{End}(\mathbb C[\mathbf x])$ by the following 
\begin{align*}
\pi(d_m)&=\sum_{j}(j-m):a_{j+1}a_{m-j}^*:+4\sum_j(j-m):a_ja_{m-j}^*: \\
	&\qquad+\sum_j(j+1-m):a_{j+1}^1a_{m-j}^{1*}:+4\sum_j(j+\frac{1}{2}-m):a_j^1a_{m-j}^{1*}: ,\\
\pi(d^1_m)&=\sum_{j}(j-m):a_{j-1}^1a_{m-j}^*: \\
	&\qquad+\sum_j(j+1-m):a_{j+1}a_{m-j}^{1*}:+4\sum_j(j+\frac{1}{2}-m):a_ja_{m-j}^{1*}: .
\end{align*}
We then have in terms of formal power series \eqnref{alpha} and \eqnref{alpha1}
\begin{align} 
\pi(d)(z):
&=P(z)\left(:\alpha(z)\partial_z\alpha^*(z): +:\alpha^1(z)\partial_z\alpha^{1*}(z):\right ) \label{eq:pi1}\\  
& \quad +\frac{1}{2}\partial_zP(z):\alpha^1(z)\alpha^{1*}(z):  \notag \\
  \pi(d^1)(z):  
	&=\, :\alpha^1(z)\partial_z\alpha^*(z):+P(z):\alpha(z)\partial_z\alpha^{1*}(z):   \label{eq:pi2}\\
	&  \quad  +\frac{1}{2}\partial_zP(z)    
	:\alpha(z)\alpha^{1*}(z):  \notag
\end{align}
 
 The following computational lemma is useful for manipulating $\lambda$-brackets of our formal distributions, we omit the proof which is routine but lengthy. Note also that similar formulae hold for other combinations of operators and formal derivatives, which are used in the proof of Proposition \ref{3ptWittrep}  and our main result, Theorem \ref{mainresult}. 
\begin{lem}\label{auxiliary lemma} The following relations hold for the elements of the oscillator algebra of Section \ref{oscilllatoralg}
\begin{enumerate} 
\item
$[:\alpha\alpha^*:{_\lambda}:\alpha\alpha^*:]=-\delta_{r,0}\lambda$,
\item
\begin{align*}
[P:\alpha\partial(\alpha^*):{_\lambda} & P:\alpha\partial(\alpha^*):] \\
& =P\left(P:\partial(\alpha^*)\alpha:\lambda+ \partial P:\partial(\alpha^*)\alpha:+P:\partial^2(\alpha^*)
\alpha:\right)\\\
&\quad +P\left(P:\alpha\partial(\alpha^*):\lambda +:\alpha\partial(\alpha^*):\partial P+P:\partial
\alpha\partial(\alpha^*):\right) \\
&\quad +P\left(\frac{1}{6}P\lambda^3+\frac{1}{2}\partial P\lambda^2+\frac{1}{2}\partial^2P\lambda\right)\delta_{r,0}\end{align*}
\item 
\begin{align*}
[P:\alpha^1\partial(\alpha^{1*}):{_\lambda} & \partial P:\alpha^1\alpha^{1*}:] \\
 &=P\partial P:\partial(\alpha^{1*})\alpha^1: +\left(P\lambda+\partial P\right)\partial P:\alpha^1 \alpha^{*1}: \\ 
&\quad + P\partial P:\partial\alpha^1 \alpha^{*1}:+ \delta_{r,0}\left(\frac{1}{2}P\lambda^2+ \partial P\lambda+\frac{1}{2}\partial^2P\right)\partial P\end{align*}
\item 
\begin{align*}
[\partial  P:\alpha^1\alpha^{1*}:{_\lambda} & P:\alpha^1\partial(\alpha^{1*}):] \\
&=
P\partial P:  \alpha^{1*}\alpha^1:\lambda +P (\partial^2 P:\alpha^{1*}\alpha^1:+\partial P:\partial (\alpha^{1*})\alpha^1:) \\
&\quad -P\partial  P:\alpha^1\partial \alpha^{1*}:  
-\delta_{r,0}P\left(\frac{1}{2}\partial P\lambda^2+\partial^2 P\lambda\right)
\end{align*}
\end{enumerate}
\end{lem}
\qed

\begin{prop} \label{3ptWittrep}  Given $\pi$ as in \eqref{eq:pi1} and \eqref{eq:pi2} we have
\begin{align*}
[\pi(d)_\lambda \pi(d)]
 &=P\partial  \pi(d)+\partial P\pi(d)+2P\pi(d)\lambda \\
&\quad +P\left(\frac{1}{3}P\lambda^3+\partial P\lambda^2+\frac{1}{2}\partial^2P\lambda\right)\delta_{r,0}  +\frac{1}{4} \delta_{r,0}(\partial P)^2\lambda \\
[\pi(d^1)_\lambda\pi(d^1)]&=\partial\pi(d)+2\pi(d)\lambda+\left(\frac{1}{3}P\lambda^3+\frac{1}{2}\partial P\lambda^2\right)\delta_{r,0},  \\
[\pi(d)_\lambda \pi(d^1)]&=P\partial \pi(d^1)  +\frac{3}{2}\partial P\pi(d^1) +2P \pi(d^1)\lambda.
\end{align*}
In particular $\pi$ is a representation of the $3$-point Witt algebra if $r=0$.
\end{prop}
\begin{proof} 

The relations follow from Lemma \ref{auxiliary lemma}, and similar calculations. For example:
\begin{align*}
[\pi(d)(z)_\lambda\pi (d)(w)]&=[\Big(P(z)\left(:\alpha(z)\partial_z\alpha^*(z): +:\alpha^1(z)\partial_z\alpha^{1*}(z):\right )+\frac{1}{2}\partial_zP(z):\alpha^1(z)\alpha^{1*}(z):\Big)_\lambda \\
&\quad  \Big(P(w)\left(:\alpha(w)\partial_w\alpha^*(w): +:\alpha^1(w)\partial_w\alpha^{1*}(w):\right )+\frac{1}{2}\partial_wP(w):\alpha^1(w)\alpha^{1*}(w):\Big) ]\\
&=[P:\alpha \partial \alpha^*:_\lambda P:\alpha \partial \alpha^*:]+[P:\alpha^1 \partial \alpha^{1*}:_\lambda P:\alpha^1 \partial \alpha^{1*}:]+\frac{1}{2}[P:\alpha^1 \partial \alpha^{1*}:_\lambda \partial P:\alpha^1\alpha^{1*}:] \\
&\quad +\frac{1}{2}[\partial P:\alpha^1\alpha^{1*}:_\lambda P:\alpha^1 \partial \alpha^{1*}: ] +\frac{1}{4}[\partial P:\alpha^1\alpha^{1*}:_\lambda \partial P:\alpha^1\alpha^{1*}:]\\ \\
&=P\left(P:\partial(\alpha^*)\alpha:\lambda+ \partial P:\partial(\alpha^*)\alpha:+P:\partial^2(\alpha^*)
\alpha:\right)\\\
&\quad +P\left(P:\alpha\partial(\alpha^*):\lambda +:\alpha\partial(\alpha^*):\partial P+P:\partial
\alpha\partial(\alpha^*):\right) \\
&\quad +P\left(\frac{1}{6}P\lambda^3+\frac{1}{2}\partial P\lambda^2+\frac{1}{2}\partial^2P\lambda\right)\delta_{r,0} \\
&\quad +P\left(P:\partial(\alpha^{1*})\alpha^1:\lambda+ \partial P:\partial(\alpha^{1*})\alpha^1:+P:\partial^2(\alpha^{1*})\alpha^1:\right)\\\
&\quad +P\left(P:\alpha^1\partial(\alpha^{1*}):\lambda +:\alpha^1\partial(\alpha^{1*}):\partial P+P:\partial
\alpha^1\partial(\alpha^{1*}):\right) \\
&\quad +P\left(\frac{1}{6}P\lambda^3+\frac{1}{2}\partial P\lambda^2+\frac{1}{2}\partial^2P\lambda\right)\delta_{r,0} \\
&\quad +\frac{1}{2}\Big(P\partial P:\partial(\alpha^{*1})\alpha^1: +\left(P\lambda+\partial P\right)\partial P:\alpha^1 \alpha^{*1}: \\ 
&\qquad + P\partial P:\partial\alpha^1 \alpha^{*1}:+ \delta_{r,0}\left(\frac{1}{2}P\lambda^2+ \partial P\lambda+\frac{1}{2}\partial^2P\right)\partial P\Big) \\
&\quad +\frac{1}{2}\Big(P\partial P:  \alpha^{1*}\alpha^1:\lambda +P (\partial^2 P:\alpha^{1*}\alpha^1:+\partial P:\partial (\alpha^{1*})\alpha^1:) \\
&\qquad -P\partial  P:\alpha^1\partial \alpha^{1*}:  
-\delta_{r,0}P\left(\frac{1}{2}\partial P\lambda^2+\partial^2 P\lambda\right)\Big) \\
&\quad -\frac{1}{4} \delta_{r,0}\left(\partial P\lambda+\partial^2P\right)\partial P \\ \\
&=P\partial  \left(P\left(:\alpha\partial\alpha^*: +:\alpha^1\partial\alpha^{1*}:\right )+\frac{1}{2}\partial P:\alpha^1\alpha^{1*}: \right) \\
&\quad +\partial P\left(P\left(:\alpha\partial\alpha^*: +:\alpha^1\partial\alpha^{1*}:\right )+\frac{1}{2}\partial P:\alpha^1\alpha^{1*}: \right) \\
&\quad +2P\left(P\left(:\alpha\partial\alpha^*: +:\alpha^1\partial\alpha^{1*}:\right )+\frac{1}{2}\partial P:\alpha^1\alpha^{1*}: \right)\lambda \\
&\quad +P\left(\frac{1}{3}P\lambda^3+\partial P\lambda^2+\frac{1}{2}\partial^2P\lambda\right)\delta_{r,0}  +\frac{1}{4} \delta_{r,0}(\partial P)^2\lambda  \\
&=P\partial  \pi(d)+\partial P\pi(d)+2P\pi(d)\lambda \\
&\quad +P\left(\frac{1}{3}P\lambda^3+\partial P\lambda^2+\frac{1}{2}\partial^2P\lambda\right)\delta_{r,0}  +\frac{1}{4} \delta_{r,0}(\partial P)^2\lambda 
\end{align*}

\end{proof}
 
\section{The 3-point Virasoro algebra}\label{3pt2cocycle}

In this section, we construct the universal central extension of the 3-point Witt algebra, which we call the 3-point Virasoro algebra. Note that the cocycles of the 3-point Witt algebra given below correspond to the ones in \cite{MR3211093} where they are given in the basis for $\text{Der}(S)$, as shown in \cite{JM}.

Recall $R=\mathbb C[t,t^{-1},u\,|\, u^2=t^2+4t]$, with basis given in \eqref{basis}.
We define cocycles $\phi_i: \text{Der}(R)\times \text{Der}(R)\to \mathbb C$ for $i= 1,2$. 

On the basis elements, for all $k,l\in\mathbb Z$ let:
\begin{align}
\phi_1(d^1_k, d^1_l)=\phi_1(t^kD,t^lD)\label{phiee}
&:=2 (l^2-l) (2 l-1)\delta_{k+l,1} +(l^3-l)\delta_{k+l,0}    \\
\phi_1(d^1_k,d_l)= \phi_1(t^kD,t^luD)&:=6(-1)^{k+l}2^{k+l}(k-1)kl  \dfrac{(2k+2l-3)!!}{ (k+l+1)!}\\
\phi_1(d_l,d^1_k)&  :=- \phi_1(d^1_k,d_l)\notag \\
\phi_1(d_k,d_l)= \phi_1(t^kuD,t^luD)&:=l (l+1) (l+2) \delta _{k+l,-2}+4 l (2 l+1) ( l+1) \delta _{k+l,-1}
 \label{phidd}\\
 &\hskip 100pt+4 l (2 l-1) (2 l+1) \delta _{k+l,0}.\notag 
\end{align}
where, by definition, $(2k+2l-3)!!=(2k+2l-3)\cdot(2k+2l-5)\cdot ... \cdot5\cdot3\cdot1$.

We extend linearly to all of $\text{Der}(R)\times \text{Der}(R)$.

Define $\phi_2 : \text{Der}(R)\times \text{Der}(R)\to \mathbb C$   on the basis elements for all $k,l\in\mathbb Z$ as
 \begin{align*}
 \phi_2(d^1_k, d^1_l)= \phi_2(t^kD,t^lD)& := -2\phi_1(t^kD,t^lD)\\
\phi_2(d_k, d_l)= \phi_2(t^kuD,t^luD)& := -2\phi_1(t^kuD,t^luD)\\ 
\phi_2(d^1_k, d_l)= \phi_2 (t^kD,t^luD) &:= -\phi_2 (t^kD,t^luD) = 0 
\end{align*}
and extend linearly  (see also  \cite{MR3211093}).

\begin{prop}The above defined $\phi_i$ are linearly independent $2$-cocycles on $\text{Der}(R)$ which are not $2$-coboundaries for all $i=1,2$.
\end{prop}

\begin{proof}
We will prove the result for  $\phi_1: \text{Der}(R)\times \text{Der}(R)\to \mathbb C$, and $\phi_2$ follows. It is easy to verify that $\phi_1$ is skew symmetric

The cocycle condition  
$
 \phi_1([a,b],c) +  \phi_1([b,c],a) +  \phi_1([c,a],b) = 0
$
for all $a,b,c \in \text{Der}(R)$, can be verified in each case. 
We use the notation of \eqref{basis}, and the commutators given in Lemma \ref{commutators}. For example
 for $m,n,r \in \mathbb Z$
\begin{align*}
\phi_1([d^1_m,d^1_n], d_r)  + &\phi_1([d^1_n,d_r], d^1_m)+ \phi_1([d_r,d^1_m], d^1_n) \\
  = &\phi_1((n-m)d_{m+n-1}, d_r) +\phi_1( ( r-n+1 )d^1_{r+n+1}, d^1_m) + \phi_1( (-4n+4r+2)d^1_{n+r}, d^1_m)\\ & + \phi_1( ( m-r-1)d^1_{r+m+1}, d^1_n)
 + \phi_1( (4m-4r-2)d^1_{m+r}, d^1_n )\\
  =& (n-m)(r(r+1)(r+2)\delta_{m+n+r, -1} + 4r(2r+1)(r+1) \delta_{m+n+r,0} \\
&\quad + 4r (2r-1)(2r+1) \delta_{m+n+r,1})\\
&\quad + ( r-n+1 )( 2(m^2-m)(2m-1)\delta_{r+n+m,0} + (m^3-m)\delta_{m+n+r,-1})\\
& \quad + (-4n+4r+2) (2(m^2-m)(2m-1)\delta_{m+n+r,1} + (m^3-m)\delta_{m+n+r,0})\\ 
&\quad +   ( m-r-1)(2 (n^2-n)(2n-1)\delta_{n+m+r,0} + (n^3-n) \delta_{n+r+m,-1})\\
& \quad+  (4m-4r-2)( 2(n^2-n)(2n-1) \delta_{m+n+r,1} + (n^3-n) \delta_{m+n+r,0} )\\
& = 0
\end{align*}
The other cases follow by similar calculations.

That $\phi_2$ is also a cocycle, and is linearly independent from $\phi_1$ is clear from the definition. 

\end{proof}


We point out the motivation for the definition of our cocycles. Recall the algebra $\text{Der}(S)$ where  $  S=\mathbb C[s,s^{-1},(s-1)^{-1} ]$, which is studied in \cite{MR3211093}.  Define $f:R\to S$ and $\phi:S\to R$ by 
\begin{equation}
f(t)=s^{-1}(s-1)^2,\quad f(u)=s-s^{-1},\quad \phi(s)=\frac{t+2+u}{2},\quad \phi(s^{-1})=\frac{t+2-u}{2}\label{isom1}
\end{equation}
 
The map $\Phi_f:\text{Der}(R)\to \text{Der}(S)$  defined by the following 
$$
\Phi_f(u)=fuf^{-1}, 
$$
is an isomorphism \cite{MR3245847}, 
and the definition of the cocycles of $\text{Der}(R)$ given above was arrived at by computing
\begin{equation}\label{eq:defn}
\phi_i(u,v):={\phi_{S}}_i(\Phi_f(u),\Phi_f(v)).
\end{equation}
on the basis elements, where ${\phi_S}_i$ are the cocycles defined in \cite{MR3211093}.  Because the cocycles obtained for $\text{Der}(S)$ are not co-boundaries, 
we have that the cocycles $\phi_1, \phi_2$ are not co-boundaries.


We define the $3$-point Virasoro algebra $\mathfrak V$ to be the universal central extension of $3$-point Witt algebra $\mathfrak W$, 
\begin{equation}\label{DefVir}
\mathfrak V=\mathfrak W\oplus \mathbb Cc_1\oplus \mathbb Cc_2
\end{equation}
where we distinguish the basis elements $\mathbf d_n$ of $\mathfrak V$, from the $d_n$ of $\mathfrak W$. The relations are 
\begin{align}
[\mathfrak V,\mathbb Cc_1\oplus \mathbb Cc_2]&=0,
\end{align}
\begin{align}
[\mathbf d_m,\mathbf d_n]&=(n-m)\left(\mathbf d_{m+n+1}+4\mathbf d_{m+n}\right)+\phi_1(d_m,d_n)c_1+\phi_2(d_m,d_n)c_2 \\
&=(n-m)\left(\mathbf d_{m+n+1}+4\mathbf d_{m+n}\right)\notag \\
&\quad-\Big(n(n+1)(n+2)\delta_{m+n,-2}+4n(n+1)(2n+1)\delta_{m+n,-1}+4n(2n-1)(2n+1)\delta_{m+n,0}\Big)\bar c\notag
\end{align}
\begin{align}
[\mathbf d^1_m,\mathbf d^1_n]&=(n-m)\mathbf d_{m+n-1}+\phi_1(d^1_m,d^1_n)c_1+\phi_2(d^1_m,d^1_n)c_2 \\
&=(n-m)\mathbf d_{m+n-1}\notag\\
&\quad+n(n-1)\Big((2n-1)\delta_{m+n,1}+(n+1)\delta_{m+n,0}\Big)\bar c\notag
\end{align}
\begin{align}
[\mathbf d_m,\mathbf d^1_n]&=(n-m-1)\mathbf d^1_{m+n+1} +(4n-4m-2)\mathbf d^1_{m+n} \\
&\hskip 100pt+\phi_1(t^muD,t^nD)c_1+\phi_2(d_m,d^1_n)c_2\notag \\
&=(n-m-1)\mathbf d^1_{m+n+1} +(4n-4m-2)\mathbf d^1_{m+n} \notag\\
&\quad+\Big(6(-1)^{k+l}2^{k+l}(k-1)kl  \dfrac{(2k+2l-3)!!}{ (k+l+1)!}\Big)c_1,\notag
\end{align}
where $\bar c=c_1-2c_2$. 

 If we set $\bar{\mathbf d}_m:=- {\mathbf d}_{m+1}$ and $\bar{\mathbf d}^1_m=-{\mathbf d^1}_{m+1}$ then for 
$$
\bar{\mathbf d}(z):=\sum_{m\in\mathbb Z}\bar{\mathbf d}_mz^{-m-2},
\quad \bar{\mathbf d}^1(z):=\sum_{m\in\mathbb Z}\bar{\mathbf d}^1_mz^{-m-2}
$$
the above defining relations become
\begin{align}
[\bar{\mathbf d}^1(z),\bar{\mathbf d}^1(w)]\label{eezw}
&=\partial_w\bar{\mathbf d}(w)\delta(z/w)+2\bar{\mathbf d}(w)\partial_w\delta(z/w) \\
&\quad -\left(P(w)\partial_w^3\delta(z/w)+\dfrac{3}{2}P'(w)\partial_w^2\delta(z/w)\right)\bar c,\notag
\end{align}
which is not far from being the relation for the Virasoro algebra.  In addition 
\begin{align}
[\bar{\mathbf d}(z),\bar{\mathbf d}(w)]\label{ddzw}
&=P(w)\partial_w \bar{\mathbf d}(w)\delta(z/w)+\partial_wP(w)\bar{ \mathbf d}(w)\delta(z/w) +2P(w) \bar{\mathbf d}(w)\partial_w\delta(z/w)\\ 
&\quad -\left(P(w)^2\partial_w^3\delta(z/w)+3P'(z)P(z)\partial_w^2\delta(z/w)+6P(z)\partial_w\delta(z/w)+12\partial_w\delta(z/w)\right)\bar c\notag
\end{align}
and  
\begin{align}
[ \bar{\mathbf d}(z),\bar{\mathbf d}^1(w)] &=P(w)\partial_w\bar{\mathbf d}^1(w)\delta(z/w)  +2P(w)\bar{\mathbf d}^1(w)\partial_w\delta(z/w)+\frac{3}{2}P'(w)\bar{\mathbf d}^1(w)\delta(z/w)\label{dezw}  \\
&\quad +\left(3w(2+w)(1+(4/w))^{1/2}\partial_w^2\delta(z/w)+w^3(1+(4/w))^{3/2} \partial_w^3\delta(z/w) 
\right)c_1 \notag
\end{align}
where we have used the following result: The Taylor series expansion of $\sqrt{1+z}$ in the formal power series ring $\mathbb C[\![z]\!]$ is 
\begin{equation}
1+\frac{z}{2}+\sum_{n\geq 2}(-1)^{n-1}\frac{(2n-3)!!}{2^nn!}z^n.
\end{equation}
In the next section below will take \eqnref{eezw}-\eqnref{dezw} as the version of the defining relations for the $3$-point Virasoro algebra.

\section{3-point Virasoro algebra action on the free field realization}

Before we go through the proof it will be fruitful to review Kac's $\lambda$-notation (see \cite{MR99f:17033} section 2.2 and \cite{MR1873994} for some of its properties) used in operator product expansions.  If $a(z)$ and $b(w)$ are formal distributions, then
$$
[a(z),b(w)]=\sum_{j=0}^\infty \frac{(a_{(j)}b)(w)}{(z-w)^{j+1}}
$$
is transformed under the {\it formal Fourier transform} 
$$
F^{\lambda}_{z,w}a(z,w)=\text{Res}_ze^{\lambda(z-w)}a(z,w),
$$
 into the sum
\begin{equation*}
[a_\lambda b]=\sum_{j=0}^\infty \frac{\lambda^j}{j!}a_{(j)}b.
\end{equation*}
So for example we have the following

\begin{lem} \label{pairs}
Given the definitions in Section \ref{threepoint}, with $\chi_1=0$  and  $\nu, \zeta$ fixed Laurent polynomials, we have
\begin{enumerate}
\item $
[\beta_\lambda\beta]=- 2  \lambda  \mathbf 1_0$,

\item $
[\beta^1_\lambda\beta]=- 2 \sqrt{P} \lambda \chi_1$,
\item $[\beta^1_\lambda\beta^1]
=-\left(2P  \lambda+\partial P  \right) \mathbf 1_0$ \\
\item \label{pbbpbb}
 \begin{align*}
[P:(\beta)^2:{_\lambda}P:\beta^2:]&=  -8 P^2\kappa_0:(\partial\beta)\beta:   -8 P^2\kappa_0:\beta^2:\lambda
-8P\partial P\kappa_0:\beta^2:  \\
&\quad +8P\left(\frac{1}{6}P\lambda^3+\frac{1}{2}\partial P\lambda^2+\frac{1}{2}\partial^2P\lambda \right) \kappa_0^2 .\notag
\end{align*}
\item \label{Pbbb1b1}
\begin{align*}
[P:\beta^2:_\lambda:(\beta^1)^2:]
&= 0=[:(\beta^1)^2:_\lambda  P:\beta^2:]
\end{align*}
\item \label{ppartialbpbb}
 \begin{equation*}
[P\partial\beta{_\lambda}:P:\beta^2:]= 4\kappa_0P^2\beta\lambda^2+8\kappa_0P\partial P\beta\lambda+4\kappa_0P\partial^2 P\beta\end{equation*}

\item \label{pbbppartialb}
\begin{align*}
[P:\beta^2:{_\lambda}P\partial\beta]&=-4\kappa_0P^2\beta\lambda^2-8\kappa_0P\partial P\beta\lambda-4\kappa_0P\partial^2 P\beta  \\
&\quad  -8\kappa_0P(P\lambda+ \partial P)\partial\beta    -4\kappa_0P^2\partial^2\beta \notag
\end{align*}
\item \label{pbbpartialpb}
\begin{align*}
[P:\beta^2:{_\lambda}\partial P\beta]&= -4\kappa_0\left(P\lambda+ \partial P\right)\partial  P\beta-4\kappa_0   \partial \beta P \partial P
\end{align*}

\item \label{partialpbpbb}
\begin{align*}
[\partial P\beta {_\lambda}P:\beta^2:]&= -4\kappa_0\left(\partial P\lambda+\partial^2 P\right)P \beta 
\end{align*}

\item \label{ppartialbppartialb}
\begin{align*}
[ P\partial \beta {_\lambda}P\partial \beta]&=2\kappa_0\left( P\lambda^3+3\partial P\lambda^2+3\partial^2 P\lambda\right) P.\end{align*}

\item \label{ppartialbpartialpb} 
 \begin{align*}
 [  P\partial \beta{_\lambda}\partial  P \beta]=2 \kappa_0\left(  P\lambda^2+2\partial P\lambda+\partial^2P\right)\partial P.
 \end{align*}

\item \label{partialpbppartialb}
 \begin{align*}
[  \partial  P \beta{_\lambda}P\partial \beta]= -2\kappa_0\left(  \partial P\lambda^2+2\partial^2P\lambda\right) P.
\end{align*}

\item \label{partialpbpartialpb}
 \begin{align*}
[  \partial  P \beta{_\lambda}\partial  P \beta]= -2\kappa_0\left( \partial P\lambda +\partial ^2P\right)\partial P.
\end{align*}

\item \label{beta12beta12}
\begin{align*}
[:(\beta^1)^2:_\lambda:(\beta^1)^2:]
&=-4\left(2P\lambda+\partial P\right) \kappa_0:(\beta^1)^2:   -8P\kappa_0:(\partial\beta^1)\beta^1:   \\ 
&\quad+8\left(\frac{1}{6}P^2\lambda^3+\frac{1}{2}P\partial P\lambda^2 +\frac{1}{4}(\partial P)^2\lambda\right) \kappa_0^2\notag
\end{align*}
\item \label{betabeta1betabeta1}
\begin{align*}
[\nu:\beta\beta^1:_\lambda:\nu\beta\beta^1:]
&=-2\kappa_0 \nu^2:(\beta^1)^2:\lambda-2\kappa_0 \nu\partial\nu:(\beta^1)^2:-2\kappa_0 \nu^2:\partial(\beta^1)\beta^1: \\
&\quad   
-\nu^2\left(2P \lambda+\partial P\right)\kappa_0:\beta^2: 
-2\kappa_0\nu\partial \nu P:\beta^2: 
-2 P \kappa_0 \nu^2:(\partial\beta)\beta:\\ 
&\quad +\frac{2}{3}\nu^2P\lambda^3\kappa_0^2 \\ 
&\quad +2\nu\left(\frac{1}{2}\nu\partial P+\partial \nu P\right)\lambda ^2\kappa_0^2 \\
&\quad +\nu(2\partial\nu\partial P +2\partial^2 \nu P )\kappa_0^2\lambda \\
&\quad +\frac{\nu}{3}\left(3\partial^2\nu\partial P+2 \partial^3\nu  P\right)\kappa_0^2 \\ \\
\end{align*}
\item \label{nubb1zpb1}

\begin{align*}
[\nu:\beta\beta^1:_\lambda \zeta \partial\beta^1]
&=-\nu\zeta\left(2P\lambda^2+3\partial P\lambda+\partial^2P\right)\kappa_0\beta \\
&\quad -\zeta
\left(4P\lambda+3\partial P\right)\kappa_0(\partial \nu\beta+\nu\partial\beta) \\
&\quad - 2P\zeta \kappa_0\left(2\partial\nu\partial\beta+\partial^2\nu\beta+\nu\partial^2\beta\right)
\end{align*}

\item \label{zpb1nbb1}
\begin{align*}
[\zeta \partial\beta^1_\lambda\nu:\beta\beta^1:]&=\zeta\nu\left(2P\lambda^2+\partial P\lambda\right) \kappa_0\beta \\
&\quad +\left(4\nu\partial \zeta P  \lambda+\nu\partial P\partial \zeta +2 P\nu\partial^2\zeta\right)\kappa_0\beta
\end{align*}

\item \label{zpb1zpb1}
\begin{align*}
[\zeta\partial\beta^1{_\lambda}\zeta\partial\beta^1]&=-2\zeta^2P\kappa_0 \lambda^3 -\zeta(3\zeta\partial P+6\partial \zeta P)\kappa_0\lambda ^2 \\
&\quad -\zeta\left(6\partial^2\zeta P+6\partial \zeta\partial P+\zeta\partial^2 P(w)\right) \kappa_0\lambda \\
&\quad -\zeta\left(\partial \zeta \partial^2P+3\partial^2\zeta \partial P+2\partial^3\zeta P\right)\kappa_0 
\end{align*}
\item \label{pbbbb1}
\begin{align*}
[P:(\beta)^2:_\lambda :\beta\beta^1:]&=  -4\kappa_0P:\beta\beta^1:\lambda  -4\kappa_0\left(\partial P:\beta\beta^1:+P:\partial \beta\beta^1:\right)\end{align*} 
\item  \label{b1b1bb1}
\begin{align*}
[:(\beta^1)^2:{_\lambda}\beta\beta^1:]&= -2\kappa_0 \left(2P:\beta^1\beta:\lambda+\partial P :\beta^1\beta:+2P:\partial \beta^1\beta: \right)
\end{align*}

\end{enumerate}
\end{lem}
\qed

Note that similar expressions hold for $\alpha^1(z)$ and $\alpha^{1*}(z)$ (the $\lambda$-notation suppresses the variables $z$ and $w$, which are understood).

We can now establish our main result. We note that the Fock space $\mathcal F$ given below is the module constructed for the 3-point affine algebra $\mathfrak{sl}(2, \mathcal R)  \oplus\left( \Omega_{\mathcal R}/d{\mathcal R}\right)$ in \cite{MR3245847}.

\begin{thm}\label{mainresult}  Suppose $\lambda,\mu,\nu,\varkappa, \chi_1,\kappa_0 \in \mathbb C$ are constants with $\kappa_0\neq 0$.  The following defines a representation of the $3$-point Virasoro algebra $\mathfrak V$ on  $\mathcal F:=\mathbb C[\mathbf x]\otimes \mathbb C[\mathbf y]\otimes \mathcal V$, with $\mathcal V$ as in \lemref{heisenbergprop}
\begin{align*}
\pi(\bar {\mathbf d})(z)&=\pi(  d)(z)+\gamma :\beta(z)^2:+\mu\partial_z\beta(z)+\gamma_1:(\beta^1(z))^2:  +\gamma_2\beta(z)\\ 
\pi(\bar{\mathbf d}^1)(z)&=\pi(  d^1)(z)+\nu :\beta(z)\beta^1(z):+\zeta\partial_z\beta^1(z)  \\
\pi(\bar{\mathbf c})&=-\Big(\frac{1}{3}\delta_{r,0}+\frac{2}{3}\nu^2\kappa_0^2-2\zeta^2\kappa_0 \Big)=-\frac{1}{3}\left(\delta_{r,0}+8\kappa _0^4 \nu ^4\right) =-\frac{1}{3}\left(\delta_{r,0}+\frac{1}{2}\right)  \\
\pi(\mathbf c_2)&=0.
\end{align*}

Where the following conditions are satisfied:
\begin{align}
\nu^2 & =\kappa_0^{-2}/4 ,   \zeta=0, \\
\gamma&=-\nu^2P(z)\kappa_0=-\frac{P(z)}{4\kappa_0},\\ \mu &=0,\quad  \gamma_1=-\nu^2\kappa_0=-\frac{1}{4\kappa_0},\quad \gamma_2=0.
\end{align}
\end{thm} 

\begin{proof}  We prove that \eqnref{eezw}--\eqnref{dezw} are satisfied by $\pi\bar{(\mathbf d)}(z)$ and $\pi\bar{(\mathbf d^1)}(z)$, the computations are presented in a compact form. 
We begin with \eqnref{eezw}:
 By \lemref{3ptWittrep},\eqnref{betabeta1betabeta1}-\eqnref{zpb1zpb1}
\begin{align*}
[\pi(\bar{\mathbf d}^1)_\lambda \pi(\bar{\mathbf d}^1)]
&=[\pi( d^1)_\lambda \pi(d^1)]-2\kappa_0 \nu^2:(\beta^1)^2:\lambda-2\kappa_0 \nu\partial\nu:(\beta^1)^2:-2\kappa_0 \nu^2:\partial(\beta^1)\beta^1: \\
&\quad   
-\nu^2\left(2P \lambda+\partial P\right)\kappa_0:\beta^2: 
-2\kappa_0\nu\partial \nu P:\beta^2: 
-2 P \kappa_0 \nu^2:(\partial\beta)\beta:\\ 
&\quad +\frac{2}{3}\nu^2P\lambda^3\kappa_0^2 
  +2\nu\left(\frac{1}{2}\nu\partial P+\partial \nu P\right)\lambda ^2\kappa_0^2 
 +\nu(2\partial\nu\partial P +2\partial^2 \nu P )\kappa_0^2\lambda \\
&\quad +\frac{\nu}{3}\left(3\partial^2\nu\partial P+2 \partial^3\nu  P\right)\kappa_0^2   
 -\nu\zeta\left(2P\lambda^2+3\partial P\lambda+\partial^2P\right)\kappa_0\beta \\
&\quad -\zeta
\left(4P\lambda+3\partial P\right)\kappa_0(\partial \nu\beta+\nu\partial\beta)  
- 2P\zeta \kappa_0\left(2\partial\nu\partial\beta+\partial^2\nu\beta+\nu\partial^2\beta\right) \\
& \quad +\zeta\nu\left(2P\lambda^2+\partial P\lambda\right) \kappa_0\beta 
  +\left(4\nu\partial \zeta P  \lambda+\nu\partial P\partial \zeta +2 P\nu\partial^2\zeta\right)\kappa_0\beta 
   -2\zeta^2P\kappa_0 \lambda^3\\
   &\quad -\zeta(3\zeta\partial P+6\partial \zeta P)\kappa_0\lambda ^2 
  -\zeta\left(6\partial^2\zeta P+6\partial \zeta\partial P+\zeta\partial^2 P\right) \kappa_0\lambda \\
&\quad -\zeta\left(\partial \zeta \partial^2P+3\partial^2\zeta \partial P+2\partial^3\zeta P\right)\kappa_0 \\ \\
&= \pi(  \partial d)+2\pi(  d)\lambda-2\kappa_0 \nu^2:(\beta^1)^2:\lambda-2\kappa_0 \nu\partial\nu:(\beta^1)^2:-2\kappa_0 \nu^2:\partial(\beta^1)\beta^1: \\
&\quad   
-\nu^2\left(2P \lambda+\partial P\right)\kappa_0:\beta^2: 
-2\kappa_0\nu\partial \nu P:\beta^2: 
-2 P \kappa_0 \nu^2:(\partial\beta)\beta:\\  
&\quad+\left(\left(4\nu\partial \zeta P  \lambda+\nu\partial P\partial \zeta +2 P\nu\partial^2\zeta\right) -\nu\zeta\left(2\partial P\lambda+\partial^2P\right) \right)\kappa_0\beta\\
&\quad -\zeta
\left(4P\lambda+3\partial P\right)\kappa_0\partial \nu\beta-
\zeta\left(4P\lambda+3\partial P\right)\kappa_0\nu\partial\beta \\
&\quad - 4\zeta P\kappa_0\partial\nu\partial\beta-2\zeta P\kappa_0\partial^2\nu\beta-2\zeta P\kappa_0\nu\partial^2\beta
\\ 
&\quad +\left(\frac{1}{3}\delta_{r,0}+\frac{2}{3}\nu^2\kappa_0^2-2\zeta^2\kappa_0\right)P \lambda^3  \\ 
&\quad +\left(\frac{1}{2}\partial P\delta_{r,0} +\left(\nu^2\partial P+2\nu\partial \nu P\right)\kappa_0  -\zeta(3\zeta\partial P+6\partial \zeta P)\right)\kappa_0\lambda ^2 \\
&\quad +\left(\nu(2\partial\nu\partial P +2\partial^2 \nu P )\kappa_0-\zeta\left(6\partial^2\zeta P+6\partial \zeta\partial P+\zeta\partial^2 P\right)\right)\kappa_0\lambda \\
&\quad +\left(\frac{\nu}{3}\left(3\partial^2\nu\partial P+2 \partial^3\nu  P\right)\kappa_0 -\zeta\left(\partial \zeta \partial^2P+3\partial^2\zeta \partial P+2\partial^3\zeta P\right)\right)\kappa_0\\ \\
\end{align*}
Assigning the values \begin{align*} 
\pi(\bar c)&=-\Big(\frac{1}{3}\delta_{r,0}+\frac{2}{3}\nu^2\kappa_0^2-2\zeta^2\kappa_0 \Big) \\
\gamma_1&=-\nu^2\kappa_0 \\
\gamma&=-\nu^2P\kappa_0 \\
\mu&=-2\nu\zeta P\kappa_0\\
\gamma_2 &= - \nu\zeta\partial P\kappa_0, \\
\end{align*}
we obtain
\begin{align*}
[\pi(\bar{\mathbf d}^1)_\lambda \pi(\bar{\mathbf d}^1)]
&= \pi(  \partial d) +2\pi(  d)\lambda-2\kappa_0 \nu^2:(\beta^1)^2:\lambda-2\kappa_0 \nu^2:\partial(\beta^1)\beta^1: \\
&\quad   
-\nu^2\left(2P \lambda+\partial P\right)\kappa_0:\beta^2: 
-2 P \kappa_0 \nu^2:(\partial\beta)\beta:\\  
&\quad-\nu\zeta\left(2\partial P\lambda+\partial^2P \right)\kappa_0\beta\\
&\quad -\zeta
\left(4P\lambda+3\partial P\right)\kappa_0\nu\partial\beta -2\zeta P\kappa_0\nu\partial^2\beta \\
&\quad +\left(\frac{1}{3}\delta_{r,0}+\frac{2}{3}\nu^2\kappa_0^2-2\zeta^2\kappa_0\right)P \lambda^3  \\ 
&\quad +\left(\frac{1}{2}\delta_{r,0} +\nu^2\kappa_0^2  -3\zeta^2\kappa_0 \right)\partial P\lambda ^2 \\ 
&=\pi(\partial\bar{\mathbf d}) +2\pi(\bar{\mathbf d})\lambda  -\left(P \lambda^3+\dfrac{3}{2}\partial P\lambda^2\right)\pi(\bar c),
\end{align*}

For \eqnref{ddzw} we have by \lemref{3ptWittrep}, and items \eqnref{pbbpbb}-\eqnref{partialpbpartialpb} in \lemref{pairs}
\begin{align*}
[\pi(\bar {\mathbf d})_\lambda\pi(\bar {\mathbf d})]
&=P\partial  \pi(d)+\partial P\pi(d)+2P\pi(d)\lambda \\
&\quad +\left(\frac{1}{3}\delta_{r,0}+\frac{8}{3}\nu^4\kappa_0^4\right)P^2\lambda^3+\left(\delta_{r,0}+8\nu^4\kappa_0^4\right)P\partial P\lambda^2+(\delta_{r,0}+8\nu^4\kappa_0^4)\left(2P   +4\right)\lambda \\
&\quad-8\nu^4\kappa_0^3 P^2:(\partial\beta)\beta:   -8\nu^4\kappa_0^3P^2:\beta^2:\lambda -8\nu^4\kappa_0^3P\partial P:\beta^2:  \\
&\quad  -8\nu^3\zeta\kappa_0^3P^2\beta\lambda^2-16\nu^3\zeta\kappa_0^3P\partial P\beta\lambda-8\nu^3\zeta\kappa_0^3P\partial^2 P\beta  \\
&\quad  -16\nu^3\zeta\kappa_0^3P(P\lambda+ \partial P)\partial\beta    -8\nu^3\zeta\kappa_0^3P^2\partial^2\beta  \\
&\quad   -4\nu^3\zeta\kappa_0^3\left(P\lambda+ \partial P\right) \partial P\beta-4\nu^3\zeta\kappa_0^3  \partial \beta P \partial P \\
&\quad +8\nu^3\zeta\kappa_0^3P^2\beta\lambda^2+16\nu^3\zeta\kappa_0^3P\partial P\beta\lambda+8\nu^3\zeta\kappa_0^3P\partial^2 P\beta    \\
&\quad +8\nu^2\zeta^2\kappa_0^3\left( P\lambda^3+3\partial P\lambda^2+3\partial^2 P\lambda\right) P
  +4\nu^2\zeta^2\kappa_0^3\left(   P\lambda^2 +2\partial P\lambda +\partial^2P\right)\partial P\\
&\quad-4\nu^4\kappa_0^3\left(2P\lambda+\partial P\right) :(\beta^1)^2:   -8P\nu^4\kappa_0^3:(\partial\beta^1)\beta^1:   \\ 
&\quad -4\nu^3\zeta\kappa_0^3\left(\partial P\lambda+\partial^2 P\right)P \beta 
 -4\nu^2\zeta^2\kappa_0^3\left(  \partial P\lambda^2+2\partial^2P\lambda \right) P 
  -2\nu^2\zeta^2\kappa_0^3\left( \partial P\lambda +\partial ^2P\right)\partial P \\ \\ 
&=P\partial  \pi(d)+\partial P\pi(d)+2P\pi(d)\lambda \\
&\quad-8\nu^4\kappa_0^3P^2:(\partial\beta)\beta:   -8 \nu^4\kappa_0^3P^2:\beta^2:\lambda
-8\nu^4\kappa_0^3P\partial P:\beta^2:  \\
&\quad-16\nu^3\zeta\kappa_0^3P(P\lambda+ \partial P)\partial\beta    -8\nu^3\zeta\kappa_0^3P^2\partial^2\beta   \\
&\quad  -4\nu^3\zeta\kappa_0^3\left(P\lambda+ \partial P\right) \partial P\beta-4\nu^3\zeta\kappa_0^3  \partial \beta P \partial P\\
&\quad -4\nu^4\kappa_0^3\left(2P\lambda+\partial P\right) :(\beta^1)^2:   -8P\nu^4\kappa_0^3:(\partial\beta^1)\beta^1:   \\ 
&\quad-4\nu^3\zeta\kappa_0^3\left(\partial P\lambda+\partial^2 P\right)P \beta \\
&\quad  +\nu^2\zeta^2\kappa_0^3\left(   8P^2\lambda^3+24P\partial P\lambda^2  +(32P  +24(P+4))\lambda  +4\partial P   \right) \\
  &\quad +\left(\frac{1}{3}\delta_{r,0}+\frac{8}{3}\nu^4\kappa_0^4\right)P^2\lambda^3+\left(\delta_{r,0}+8\nu^4\kappa_0^4\right)P\partial P\lambda^2+(\delta_{r,0}+8\nu^4\kappa_0^4)\left(2P   +4\right)\lambda
\end{align*}

Set $\nu^2\kappa_0^2=\frac{1}{4}$, $\zeta=0$. Then $\mu=0$, $\gamma_2=0$ and $\pi(\bar c)=-\Big(\frac{1}{3}\delta_{r,0}+\frac{1}{6}\Big)$.
We obtain
  
  \begin{align*}
[\pi(\bar {\mathbf d})_\lambda\pi(\bar {\mathbf d})]
&=P\partial\pi(  d)+\partial P\pi(  d)+2P \pi(  d)\lambda  \\
&\quad -\left(P^2\lambda^3+3P\partial P\lambda^2+6P\lambda+12\lambda\right)\pi(\bar c)  \\ \\
&\quad +2 P\gamma :\partial\beta\beta:+2P\gamma :\beta^2:\lambda +2\gamma \partial P:\beta^2: \\
&\quad +2P\gamma_1:\beta^1\partial\beta^1 :  +2P\gamma_1:(\beta^1)^2:\lambda +\gamma_1\partial P:(\beta^1)^2:\\ 
&=P\partial\pi(  d)+2 P^2\gamma_1 :\partial\beta\beta: +P\partial P\gamma_1 :\beta^2:+2P\gamma_1:\beta^1\partial\beta^1 :\\
&\quad+\partial P\pi(  d) +P\partial P\gamma_1 :\beta^2: +\gamma_1\partial P:(\beta^1)^2:\\
&\quad+2P \pi(  d)\lambda +2P^2\gamma_1 :\beta^2:\lambda  +2P\gamma_1:(\beta^1)^2:\lambda\\
&\quad -\left(P^2\lambda^3+3P\partial P\lambda^2+6P\lambda+12\lambda\right)\pi(\bar c)  \\ \\
&=P\partial \pi(\bar{\mathbf d})+\partial P\pi(\bar{ \mathbf d}) +2P \pi(\bar{\mathbf d})\lambda \\ 
&\quad -\left(P^2\lambda^3+3P\partial P\lambda^2+6P\lambda+12\lambda\right)\pi(\bar c).
\end{align*} 

For \eqnref{dezw} we have 
\begin{align*}
[\pi(\bar {\mathbf d})_\lambda\pi(\bar {\mathbf d^1})]
&=[\pi(  d)_\lambda \pi(d^1)]+\nu[\gamma:\beta^2:_\lambda :\beta\beta^1:] +\gamma_1\nu[:(\beta^1)^2{_\lambda} :\beta\beta^1:]   \\ \\
&=P\partial \pi(d^1)  +\frac{3}{2}\partial P\pi(d^1) +2P \pi(d^1)\lambda \\
&\quad+ 4\kappa_0^2\nu^3 P:\beta\beta^1:\lambda +4\kappa_0^2\nu^3\left(\partial P:\beta\beta^1:+P:\partial \beta\beta^1:\right)  \\
&\quad +2\kappa_0^2\nu^3\left(2P:\beta^1\beta:\lambda+\partial P :\beta^1\beta:+2P:\partial \beta^1\beta: \right) \\ \\
&=P\partial \pi(d^1)  +\frac{3}{2}\partial P\pi(d^1) +2P \pi(d^1)\lambda \\
&\quad+  \nu P:\beta\beta^1:\lambda +\nu\left(\partial P:\beta\beta^1:+P:\partial \beta\beta^1:\right)  \\
&\quad +\frac{\nu}{2}\left(2P:\beta^1\beta:\lambda+\partial P :\beta^1\beta:+2P:\partial \beta^1\beta: \right) \\ \\
&=P\partial\pi(  e)+\nu P :\partial\beta\beta^1:+\nu P :\beta\partial\beta^1:  \\ 
&\quad +2P \pi(  e)\lambda+2\nu P : \beta\beta^1:\lambda  
  +\frac{3}{2}\partial P\pi(  e)+ \frac{3}{2}\nu\partial P :\beta\beta^1: \\
&=P\partial\pi(\bar{\mathbf d}^1)  +2P\pi(\bar{\mathbf d}^1)\lambda+\frac{3}{2}\partial P\pi(\bar{\mathbf d}^1)
\end{align*}

\end{proof}

We note that the values of the expressions $\pi(\bar c), 
\gamma_1,
\gamma,
\mu,$ and 
$\gamma_2$ chosen here are sufficient to have a representation on our chosen Fock space, and it is possible that other values could appear for other representations. 
We  conjecture that the semi-direct product algebra $\mathfrak V \ltimes \hat{\mathfrak g}$ will act on the same Fock space given appropriate conditions. 

\section{Appendix}
For the convenience of the reader, we include the following results which are useful for performing the computations necessary for proving our results.

\begin{thm}[Wick's Theorem, 
\cite{MR99f:17033} ]  Let  $a^i(z)$ and $b^j(z)$ be formal
distributions with coefficients in the associative algebra 
 $\End(\mathbb C[\mathbf x]\otimes \mathbb C[\mathbf y])$, 
 satisfying
\begin{enumerate}
\item $[ \lfloor a^i(z)b^j(w)\rfloor ,c^k(x)_\pm]=[ \lfloor
a^ib^j\rfloor ,c^k(x)_\pm]=0$, for all $i,j,k$ and
$c^k(x)=a^k(z)$ or
$c^k(x)=b^k(w)$.
\item $[a^i(z)_\pm,b^j(w)_\pm]=0$ for all $i$ and $j$.
\item The products 
$$
\lfloor a^{i_1}b^{j_1}\rfloor\cdots
\lfloor a^{i_s}b^{i_s}\rfloor:a^1(z)\cdots a^M(z)b^1(w)\cdots
b^N(w):_{(i_1,\dots,i_s;j_1,\dots,j_s)}
$$
have coefficients in
$\End(\mathbb C[\mathbf x]\otimes \mathbb C[\mathbf y])$ for all subsets
$\{i_1,\dots,i_s\}\subset \{1,\dots, M\}$, $\{j_1,\dots,j_s\}\subset
\{1,\cdots N\}$. Here the subscript
${(i_1,\dots,i_s;j_1,\dots,j_s)}$ means that those factors $a^i(z)$,
$b^j(w)$ with indices
$i\in\{i_1,\dots,i_s\}$, $j\in\{j_1,\dots,j_s\}$ are to be omitted from
the product
$:a^1\cdots a^Mb^1\cdots b^N:$ and when $s=0$ we do not omit
any factors.
\end{enumerate}
Then
\begin{align*}
:&a^1(z)\cdots a^M(z)::b^1(w)\cdots
b^N(w):= \\
  &\sum_{s=0}^{\min(M,N)}\sum_{i_1<\cdots<i_s,\,
j_1\neq \cdots \neq j_s}\lfloor a^{i_1}b^{j_1}\rfloor\cdots
\lfloor a^{i_s}b^{j_s}\rfloor
:a^1(z)\cdots a^M(z)b^1(w)\cdots
b^N(w):_{(i_1,\dots,i_s;j_1,\dots,j_s)}.
\end{align*}
\end{thm}

\begin{thm}[Taylor's Theorem, \cite{MR99f:17033}, 2.4.3]
\label{Taylorsthm}  Let
$a(z)$ be a formal distribution.  Then in the region $|z-w|<|w|$,
\begin{equation}
a(z)=\sum_{j=0}^\infty \partial_w^{(j)}a(w)(z-w)^j.
\end{equation}
\end{thm}

\begin{thm}[\cite{MR99f:17033}, Theorem 2.3.2]\label{kac}  Set $\mathbb C[\mathbf x]=\mathbb C[x_n,x^1_n|n\in\mathbb Z]$ and $\mathbb C[\mathbf y]= C[y_m,y_m^1|m\in\mathbb N^*]$.  Let $a(z)$ and $b(z)$ 
be formal distributions with coefficients in the associative algebra 
 $\End(\mathbb C[\mathbf x]\otimes \mathbb C[\mathbf y])$ where we are using the usual normal ordering.   The
following are equivalent
\begin{enumerate}[(i)]
\item
$\displaystyle{[a(z),b(w)]=\sum_{j=0}^{N-1}\partial_w^{(j)}
\delta(z-w)c^j(w)}$, where $c^j(w)\in \End(\mathbb C[\mathbf x]\otimes \mathbb 
C[\mathbf y])[\![w,w^{-1}]\!]$.
\item
$\displaystyle{\lfloor
ab\rfloor=\sum_{j=0}^{N-1}\iota_{z,w}\left(\frac{1}{(z-w)^{j+1}}
\right)
c^j(w)}$.
\end{enumerate}\label{Kacsthm}
\end{thm}


\section{Acknowledgement}
The first two authors would like to thank The College of Charleston mathematics department for summer research support during  this project.
The third author would like to thank the other two authors and the College of Charleston for the hospitality during his visit in 2014. The third author was partially supported by FAPESP grant (2012/02459-8).

\bibliographystyle{alpha}

\def\cprime{$'$} \def\cprime{$'$} \def\cprime{$'$}

 \end{document}